\documentclass[a4paper,12pt,leqno]{amsart}
\usepackage{latexsym}
\usepackage[all]{xy}
\usepackage{amssymb,bm} 
\usepackage{amsmath} 
\usepackage{tikz}
\usetikzlibrary{arrows,cd}  
\usepackage{amsthm}

\usepackage{CJKutf8} 

\def\Z{{\mathbb{Z}}}
\def\K{{\mathbb{K}}}
\def\A{{\mathcal{A}}}

\DeclareMathOperator{\rank}{rank}

\DeclareMathOperator{\Der}{Der}

\DeclareMathOperator{\pd}{pd}
\DeclareMathOperator{\depth}{depth}

\DeclareMathOperator{\Hom}{Hom}

\DeclareMathOperator{\supp}{supp}  
\DeclareMathOperator{\reg}{reg}  
\newcommand{\mideal}{\ensuremath{\mathfrak{m}}} 
\newcommand{\vertex}{\ensuremath{\mathrm{Vert}}} 

\numberwithin{equation}{section}

\newtheorem{theorem}{Theorem}[section]
\newtheorem{prop}[theorem]{Proposition}
\newtheorem{cor}[theorem]{Corollary}
\newtheorem{lemma}[theorem]{Lemma}

\theoremstyle{definition}

\newtheorem{example}[theorem]{Example}
\theoremstyle{remark}
\newtheorem{rem}[theorem]{Remark}

\title{Vector fields of graphic arrangements and\\
face rings of simplicial posets}

\author{Takuro Abe}
\address{
Takuro Abe,
Department of Mathematics,
Rikkyo University, 3-34-1 Nishi Ikebukuro, Toshima-ku, 1718501 Tokyo,
Japan.
}
\email{abetaku@rikkyo.ac.jp}

\author{Satoshi Murai}
\address{
Satoshi Murai,
Department of Mathematics
Faculty of Education
Waseda University,
1-6-1 Nishi-Waseda, Shinjuku, Tokyo 169-8050, Japan}
\email{s-murai@waseda.jp}

\begin{document}

\begin{abstract}
A graphic arrangement $\A_G$ associated with a simple graph $G$ is a classical and well-studied object in the theory of hyperplane arrangements.
In this note, we show that, for a connected graph $G$,
a slight modification of the logarithmic vector field $D(\A_G)$ of $\A_G$ is isomorphic to the face ring of a certain simplicial poset.
This allows us to give formulas for several algebraic invariants of $D(\A_G)$, such as its Hilbert series, local cohomology, projective dimension, and Castelnuovo--Mumford regularity,
in terms of combinatorial and topological information about the corresponding simplicial poset.
As a by-product, we also give an explicit vector space basis of $D(\A_G)$.
\end{abstract}

\maketitle

\section{Introduction}
The logarithmic vector field $D(\A)$ of a hyperplane arrangement $\A$ is one of the central objects in the algebraic study of hyperplane arrangements.
A major theme in the study of logarithmic vector fields is their freeness, and the freeness of hyperplane arrangements has been studied extensively.
For a free arrangement $\A$, 
the module structure is comparatively simple, since $D(\A)$ is free, 
whereas for non-free arrangements it is generally difficult to describe the module structure of $D(\A)$ explicitly.
In this note, we give a concrete combinatorial description of  $D(\A_G)$ for graphic arrangements $\A_G$,
which form one of the most fundamental classes of arrangements,
by establishing a connection between $D(\A_G)$ and another object in algebraic combinatorics, namely face rings of simplicial posets.
Our result enables us to study algebraic properties of $D(\A_G)$ from combinatorial and topological information about simplicial posets using ideas from Stanley--Reisner theory.

Let us introduce the necessary notation.
Let $\K$ be a field,
$V=\K^\ell$ and $S=\K[x_1,\dots,x_\ell]$ the polynomial ring over $\K$.
For a linear hyperplane $H$ of $V$, we write $\alpha_H$ for a defining linear form of $H$,
and for a linear form $\alpha$, we write $H_\alpha$ for the linear hyperplane of $V$ defined by $\alpha$.
Let $\Der(S)=\bigoplus_{i=1}^\ell S \cdot \partial_{x_i}$ be the module of derivations of $S$.
The {\bf logarithmic vector field} $D(\A)$ of a linear hyperplane arrangement $\A$ in $V$ is the $S$-module
\[D(\A)= \{ \theta \in \Der(S) \mid \theta \alpha_H \in (\alpha_H) \text{ for all }H \in \A\}.\]

We next recall the face ring of a simplicial poset, introduced by Stanley \cite{St} as a generalization of the Stanley--Reisner ring of a simplicial complex.
A {\bf simplicial poset} $P$ is a finite poset with a minimal element $\hat 0$ such that, for every $\sigma \in P$, the interval $[\hat 0, \sigma]=\{\tau \in P \mid \hat 0 \leq \tau \leq \sigma \}$ is a Boolean algebra.
We write $\rank(\sigma)=k$ if the Boolean algebra $[\hat 0,\sigma]$ has rank $k$.
Every simplicial poset $P$ is a CW-poset, that is, the face poset of some regular CW-complex $\Gamma(P)$ \cite{Bj84}.
Geometrically, simplicial posets generalize simplicial complexes. Their corresponding CW-complexes may be viewed as regular CW-complexes whose cells are simplices, but unlike simplicial complexes, two simplices are allowed to intersect in an arbitrary subcomplex of their boundaries rather than in a single common face.
Let $P$ be a simplicial poset and let $A=\K[x_\sigma \mid \sigma \in P]$ with the grading $\deg x_\sigma=\rank \sigma$.
If $\sigma,\tau \in P$ have a common upper bound,
then $\sigma,\tau$ have a meet, that is,
a unique largest lower bound.
In this case, we write $\sigma \wedge \tau$ for the meet of $\sigma$ and $\tau$.
The {\bf face ring} $\K[P]$ of $P$ is the quotient ring $A/I_P$,
where $I_P$ is the graded ideal of $A$ generated by the following elements:
\begin{itemize}
    \item $x_{\hat 0}-1$;
    \item $x_\sigma x_\tau$ for all pairs $\sigma,\tau \in P$ that have no common upper bound;
    \item $x_\sigma x_\tau-x_{\sigma \wedge \tau} (\sum_\rho x_\rho)$ for all pairs $\sigma,\tau \in P$ that have a common upper bound, where $\rho$ runs over all minimal common upper bounds of $\sigma$ and $\tau$.
\end{itemize}

We now state our main result.
Let $G$ be a simple graph with vertex set $[\ell]=\{1,2,\dots,\ell\}$.
The {\bf graphic arrangement} $\A_G$ of $G$ is the arrangement  defined by
\[\A_G=\{H_{x_u-x_v} \subset \K^\ell \mid \{u,v\} \text{ is an edge of }G\}.\]
Let $R=\K[x_i-x_j\mid 1 \leq i < j \leq \ell]\subset S$.
Instead of $D(\A_G)$, we consider the module
\[D_R(\A_G) = D(\A_G) \cap \left( \bigoplus_{i=1}^\ell R \cdot \partial_{x_i}\right).\]
This modification is natural for graphic arrangements, since their defining equations involve only the differences $x_i-x_j$; moreover, no information is lost, since elements of $D_R(\A_G)$ generate $D(\A_G)$.
Define the poset $P_G$ by\footnote{There are several known ways to define simplicial posets from graphs \cite{BK,FGG,MMP}. Our definition looks similar to the constructions in \cite{BK,FGG} but is slightly different.}
\[P_G=\big\{\big(V(C),T\big) \mid T \subsetneq [\ell], \text{ $C$ is a connected component of $G-T$}\big\},\]
where $G-T$ is the graph obtained from $G$ by removing the vertices in $T$
and $V(C)$ is the set of vertices of $C$,
with the partial order
\[
(W,T) \leq (W',T') \Leftrightarrow
W \supset W' \text{ and } T \subset T'.
\]
Thus an element of $P_G$ records a set of deleted vertices together with one connected component that remains after the deletion.
The order relation reflects deleting more vertices and passing to smaller connected components.
See Figure \ref{fig1}.
As we explain later in Lemma \ref{graphposet}, if $G$ is connected then
the poset $P_G$ is a simplicial poset with minimal element $([\ell],\emptyset)$.
We regard $\K[P_G]$ as an $S$-module by identifying $x_i$ with $\sum_{(W,T) \in P_G,\ \! T=\{i\}} x_{(W,T)}$.
Here is our main result.

\begin{figure}
    \centering
    \includegraphics[width=0.9\linewidth]{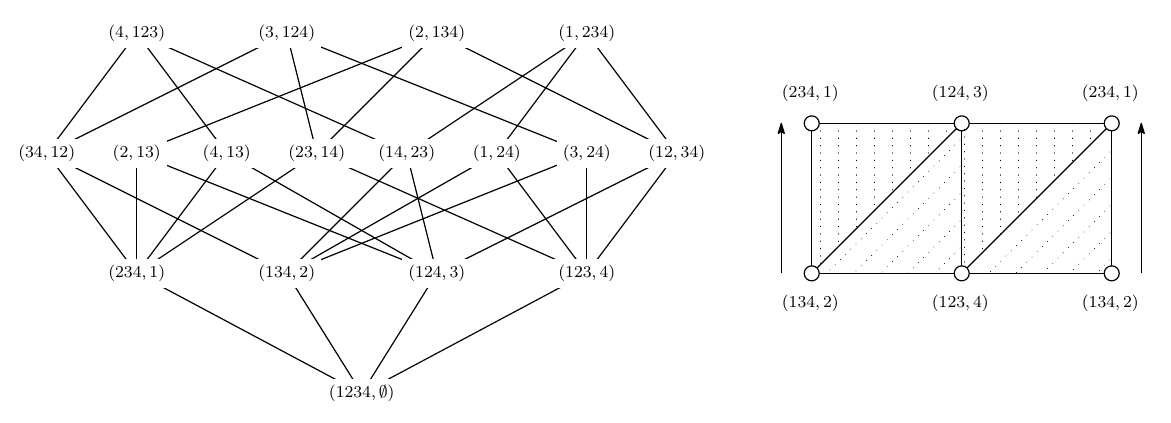}
    \caption{The simplicial poset $P_G$ (left) and the corresponding regular CW-complex $\Gamma(P_G)$ (right) for the graph $G=$\raisebox{-2mm}{\includegraphics[width=0.06\linewidth]{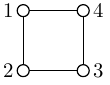}}.}
    \label{fig1}
\end{figure}

\begin{theorem}
\label{thm:1.1}
If $G$ is a connected graph with vertex set $[\ell]$, then
\[D_R(\A_G) \cong \K[P_G]\]
as $R$-modules.
\end{theorem}

Theorem \ref{thm:1.1} gives a concrete combinatorial description of $D(\A_G)$ in terms of $P_G$ (see Remark \ref{structureS}).
Moreover, many algebraic properties of the face ring $\K[P]$ can be computed from combinatorial and topological information about a simplicial poset $P$.
Theorem \ref{thm:1.1} enables us to apply such methods to study algebraic properties of $D(\A_G)$.
As applications of Theorem \ref{thm:1.1}, we prove the following results.
\begin{itemize}
    \item We give an explicit $\K$-basis for $D(\A_G)$ (Proposition \ref{prop:Kbasis}). This reproves the authors' recent result \cite[Corollary 7.13]{AM}, which gives a combinatorial formula for the Hilbert series of $D(\A_G)$.
    \item We give a formula for the Hilbert series of the local cohomology of $D(\A_G)$ in terms of the homology of links of $P_G$ (Corollary \ref{cor:localcohomology}).
    This leads to formulas for the projective dimension and regularity of $D(\A_G)$ (Corollary \ref{cor:pdandreg}).
    \item We prove that if $G$ is connected, then $P_G$ is a normal pseudomanifold, and that the freeness of $\A_G$ is equivalent to $\Gamma(P_G)$ being homeomorphic to a ball or a sphere (Propositions \ref{pmanifold} and \ref{freeness}). 
\end{itemize}

Finally, let us explain where the isomorphism in Theorem \ref{thm:1.1} comes from.
In our recent work on the module of logarithmic differential $1$-forms \cite[\S 7.4]{AM}, we show
\begin{align}
\label{dual}
D_R(\A_G) \cong H_{\mideal_S}^{\ell-1}\big(J(G^c)/(x_1\cdots x_\ell)\big)^\vee
\end{align}
as $R$-modules,
where $J(G^c)$ is a squarefree monomial ideal known as the cover ideal of the complement graph $G^c$ of $G$,
$H^*_{\mideal_S}(-)$ denotes the local cohomology module w.r.t.\ the maximal ideal $\mideal_S$ of $S$, and $M^\vee$ is the Matlis dual of an $S$-module $M$.
There is a combinatorial way to determine the structure of the right-hand side of \eqref{dual} (see \cite{Gr} and \cite[Theorem 1.8]{AS}), and this description led us to notice its resemblance to the face ring of a simplicial poset.
We do not use \eqref{dual} to prove Theorem \ref{thm:1.1}. Indeed, we prove the theorem by explicitly constructing an isomorphism using a generating set of $D(\A_G)$ recently obtained by M\"uhlherr \cite{Mu}, and our proof itself is elementary.
However, 
we would not have imagined that the two objects appearing in Theorem \ref{thm:1.1} could be isomorphic without this abstract observation.

The organization of this paper is as follows. In \S2 we prove Theorem \ref{thm:1.1}. In \S3 we give some applications of Theorem \ref{thm:1.1}.
\bigskip

\noindent
\textbf{Acknowledgments}:
The first author
is partially supported by JSPS KAKENHI Grant Numbers JP23K17298 and JP25K24692.
The second author is partially supported by JSPS KAKENHI Grant Numbers JP25K06943.
A connection between the $f$-vector of $P_G$ and the subgraph component polynomial in Remark \ref{rem3.4} was pointed out by ChatGPT (OpenAI, March 2026).

\section{Proof of Theorem \ref{thm:1.1}}

In this section, we prove our main result.
We construct an isomorphism between $\K[P_G]$ and $D_R(\A_G)$ by exhibiting compatible $\K$-bases for these two objects.

\subsection*{A basis for $\K[P_G]$}
First, we discuss a $\K$-basis for the face ring of a simplicial poset.
Let $P$ be a simplicial poset.
For an element $\sigma \in P$, we write $\rank(\sigma)=k$ if $[\hat 0,\sigma]$ is a Boolean algebra of {\bf rank} $k$.
The maximum of the ranks of the elements in $P$ is called the rank of $P$.
We call a rank $1$ element of $P$ a {\bf vertex} of $P$
and write $P_1=\{v \in P \mid \rank (v) =1\}$ for the set of vertices of $P$.
Also, for $\sigma \in P$, we call an element $v\in P_1$ with $v \leq \sigma$ a vertex of $\sigma$ and we denote the set of vertices of $\sigma$ by $\vertex(\sigma)$.
Note that if $\sigma$ has rank $k$ then $\sigma$ has exactly $k$ vertices.
The following $\K$-basis for $\K[P]$ was obtained by Duval \cite[Proposition 3.1]{Duval}.

\begin{lemma}[Duval]
\label{lem:duval}
Let $P$ be a simplicial poset. The set
\[ \biguplus_{\sigma \in P} \big\{x_\sigma m \mid m \text{ is a monomial in }\{x_v \mid v \in \vertex(\sigma)\}\big\}\]
is a $\K$-basis of $\K[P]$.
\end{lemma}

\begin{example}
Consider the simplicial poset $P=\{\hat 0, v_1,v_2,e_1,e_2\}$ shown below.
\begin{center}
\includegraphics[width=0.1\linewidth]{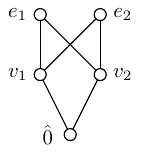}
\end{center}
Then its face ring is given by
\begin{align*}
\K[P]
&=\K[x_{\hat 0},x_{v_1},x_{v_2},x_{e_1},x_{e_2}]/(x_{\hat 0}-1,x_{v_1}x_{v_2}-x_{\hat 0} (x_{e_1}+x_{e_2}),x_{e_1}x_{e_2})\\
&\cong \K[x_{v_1},x_{v_2},x_{e_1},x_{e_2}]/(x_{v_1}x_{v_2}- (x_{e_1}+x_{e_2}),x_{e_1}x_{e_2})
\end{align*}
and it has a $\K$-basis
\begin{align*}
&\{x_{\hat 0}\} \uplus \{x_{v_1}x_{v_1}^a\mid a \in \Z_{\geq 0}\}
\uplus \{x_{v_2}x_{v_2}^a\mid a \in \Z_{\geq 0}\}\\
&\uplus \{x_{e_1}x_{v_1}^ax_{v_2}^b\mid a,b \in \Z_{\geq 0}\}
\uplus \{x_{e_2}x_{v_1}^ax_{v_2}^b\mid a,b \in \Z_{\geq 0}\}.
\end{align*}
\end{example}

Next, we consider the simplicial poset $P_G$ defined in the introduction.
Recall that a (finite simple) graph $G=(V(G),E(G))$ is a pair consisting of a finite set $V(G)$ and a family $E(G)$ of two-element subsets of $V(G).$
An element of $V(G)$ is called a {\bf vertex} of $G$ and an element of $E(G)$ is called an {\bf edge} of $G$.
The set $V(G)$ is called the {\bf vertex set} of $G$.
For a subset $W \subset V(G)$,
the {\bf induced subgraph} $G[W]$ of $G$ on $W$ is the graph whose vertex set is $W$ and whose edge set is
$\{ \{u,v\} \in E(G) \mid u,v \in W\}$.
Similarly, for $T \subset V(G)$, we write $G-T$ for the graph obtained from $G$ by removing vertices in $T$;
in other words, $G-T$ is the induced subgraph of $G$ on $V(G)\setminus T$.
We say that $T \subset [\ell]$ is a {\bf separator} of $G$ if $G-T$ is not connected.
A graph $G$ is {\bf $k$-connected} if $|V(G)|>k$ and $G-T$ is connected for every subset $T \subset V(G)$ of size $<k$.
We write $\mathrm{Comp}(G)$ for the set of all connected components of $G$.
The following lemma follows directly from the definition of $P_G$.

\begin{lemma}
\label{graphposet}
Let $G$ be a graph with vertex set $[\ell]$ and $(W,T) \in P_G$.
\begin{itemize}
    \item [(1)] For every $T'\subset T$, there is a unique subset $W_{T'} \supset W$ such that $(W_{T'},T') \in P_G$.
\item[(2)]
If $G$ is connected then $P_G$ is a simplicial poset of rank $\ell-1$.
\end{itemize}
\end{lemma}

\begin{proof}
For (1), let $C$ be the connected component of $G-T$ with $V(C)=W$.
For every $T' \subset T$, there is a unique connected component $C_{T'}$ of $G-T'$ that contains $C$.

Assume $G$ is connected.
Then $P_G$ has the minimal element $\hat 0=([\ell],\emptyset)$. 
Also, for $(W,T) \in P_G$, by part (1) the interval $[\hat 0,(W,T)]$
is precisely $\{(W_{T'},T') \mid T' \subset T\}$, which is a Boolean algebra of rank $|T|$.
\end{proof}

The $\K$-basis given in Lemma \ref{lem:duval} admits the following simpler description for $P_G$.
Let $G$ be a connected graph with vertex set $[\ell]$.
For $\sigma=(W,T) \in P_G$, we call $T$ the {\bf support} of $\sigma$ and write $\supp(\sigma)=T.$
Let 
\[X_i= \sum_{\sigma \in P_G,\ \! \supp(\sigma)=\{i\}} x_\sigma\]
for $i=1,2,\dots,\ell$.
For a non-negative integer vector $\bm a=(a_1,\dots,a_\ell)$, we write 
\[X^{\bm a}=X_1^{a_1}X_2^{a_2} \cdots X_\ell^{a_\ell}\]
and we call
\[
\supp(\bm a)=\{i \mid a_i>0\}
\]
the {\bf support} of $\bm a$.
For $T \subset [\ell]$,
we write
\[\Z^\ell_T=\{ \bm a \in \Z_{\geq 0}^\ell \mid \supp(\bm a) \subset T\}.\]

\begin{lemma}
\label{lem:Pbasis}
With the same notation as above, the set
\[
\mathrm{Base}(P_G)=\{x_{(W,T)} X^{\bm a} \mid (W,T) \in P_G,\ \bm a \in \Z_T^\ell\}\]
is a $\K$-basis of $\K[P_G]$.
\end{lemma}

\begin{proof}
Fix a rank $r$ element $\sigma=(W,T) \in P_G$
and write $T=\{t_1,\dots,t_r\}$.
Let 
\[\vertex(\sigma)=\big\{ v_1=(W_1,\{t_1\}),v_2=(W_2,\{t_2\}),\dots,v_r=(W_r,\{t_r\})\big\}.\]
Fix $k$ with $1 \leq k \leq r$.
By Lemma \ref{lem:duval},
to prove the desired statement, it is enough to prove
\[x_\sigma x_{v_k}=x_\sigma X_{t_k}.\]
Write
\[X_{t_k}=x_{v_k}+x_{u_1}+\cdots+x_{u_m}.\]
Note that
$u_1,\dots,u_m$ are the other vertices of $P_G$ with support $\{t_k\}$.
If $\sigma$ and $u_i$ have a common upper bound $\tau$ then $\tau$ has two distinct vertices $v_k$ and $u_i$ having the same support,
which is impossible by Lemma \ref{graphposet}(1).
Hence $x_\sigma x_{u_i}=0$ for all $i$ and this proves
\[x_\sigma X_{t_k}=x_\sigma x_{v_k}\]
for $k=1,2,\dots,r$, as desired.
\end{proof}

Let $\bm e_1,\dots,\bm e_\ell$ be the standard vectors of $\Z^\ell$.
Let $x_{(W,T)}X^{\bm a} \in \mathrm{Base}(P_G)$ and $i \in [\ell]$.
Then $X_i(x_{(W,T)}X^{\bm a})$ can be written as a linear combination of the elements in $\mathrm{Base}(P_G)$ in $\K[P_G]$.
This is described in the following lemma.

\begin{lemma}
    \label{lem:obvious}
Let $G$ be a connected graph with vertex set $[\ell]$,
$(W,T) \in P_G$, $\bm a \in \Z_T^\ell$, and let $C$ be the connected component of $G-T$ such that $V(C)=W$.
    \begin{enumerate}
        \item [(1)] If $i \in T$, then $X_i(x_{(W,T)} X^{\bm a})=x_{(W,T)} X^{\bm a+\bm e_i} \in \mathrm{Base}(P_G)$.
        \item [(2)] If $i \in W$, then one has
        \begin{align} \label{kakeru} X_i(x_{(W,T)}X^{\bm a})= \sum_{C' \in \mathrm{Comp}(C-\{i\})} x_{(V(C'),T\cup \{i\})}X^{\bm a} \ \ \text{ in $\K[P_G]$.}
        \end{align}
        \item[(3)] If $i \in [\ell]\setminus(T \cup W)$, then $(W,T\cup \{i\}) \in P_G $ and 
        $ X_i(x_{(W,T)}X^{\bm a})=x_{(W,T\cup \{i\})}X^{\bm a}$ in $\K[P_G]$.
    \end{enumerate}
\end{lemma}

\begin{proof}
Part (1) is obvious. We prove (2) and (3).
Assume $i \not \in T$.
Let
\[\Xi=\big\{(W',T')\in P_G \mid (W',T') \geq (W,T) \text{ and }T'=T \cup \{i\}\big\}.
\]
By Lemma \ref{graphposet},
for $\sigma,\tau \in P_G$,
if $\rho$ is a minimal common upper bound of $\sigma$ and $\tau$ then
\[
\mathrm{supp}(\rho)=\mathrm{supp}(\sigma) \cup \mathrm{supp}(\tau)\ \ \ 
\mbox{ and } \ \ \
\mathrm{supp}(\sigma \wedge \tau)=\mathrm{supp}(\sigma) \cap \mathrm{supp}(\tau).
\]
Recall that $([\ell],\emptyset)$ is the only element of $P_G$ whose support is $\emptyset$ and $x_{([\ell],\emptyset)}=1$ in $\K[P_G]$.
Then, for every vertex $u=(U,\{i\})$ with support $\{i\}$,
the meet $(W,T) \wedge u$ is the minimal element $([\ell],\emptyset)$,
and  in $\K[P_G]$ the product $x_u x_{(W,T)}$ is 
either zero or a sum of variables $x_\tau$ with $\tau \in \Xi$ by the definition of $I_{P_G}$ (consider a generator of $I_{P_G}$ of the third type).
Also, for each $(W',T') \in \Xi$, there is a unique vertex $v$ such that $x_{(W',T')}$ appears in the product $x_v x_{(W,T)}$ (indeed, $v$ is the vertex of $(W',T')$ with support $\{i\}$).
Since $X_i$ is the sum of the variables $x_u$ with $\supp(u)=\{i\}$,
this proves
\[
X_i\ \! x_{(W,T)}=\sum_{(W',T') \in \Xi} x_{(W',T')}\ \ \text{ in }\K[P_G].
\]
On the other hand, by the definition of $P_G$,
\[
\Xi=
\begin{cases}
    \big\{(V(C'),T\cup\{i\}) \mid  C' \in \mathrm{Comp}(C-\{i\})\big\} & \text{ if }i \in W,\\
\big\{(V(C),T \cup \{i\})\big\}  & \text{ if } i \notin W.
\end{cases}
\]
This proves the desired statement.
\end{proof}

\begin{rem}
\label{structureS}
One may feel that the definition of face ring of simplicial posets is complicated, but the $\K[X_1,\dots,X_\ell]$-module structure of $\K[P_G]$ is relatively simple.
Lemmas \ref{lem:Pbasis} and \ref{lem:obvious} explain how the combinatorial data of $P_G$ determine such a module structure.
\end{rem}

\subsection*{A basis for $D_R(\A_G)$}
Next, we give a $\K$-basis for $D_R(\A_G)$ that is analogous to the basis in Lemma \ref{lem:Pbasis}.
For a subset $T \subsetneq [\ell]$ and $W \subset [\ell]\setminus T$,
we define $\theta_{W,T}$ by
\[
\theta_{W,T}= \sum_{v \in W} \left( \prod_{t \in T} (x_t-x_v) \right) \partial_{x_v} \in \Der(S).\]
Also, for a vector $\bm a=(a_1,\dots,a_\ell) \in \Z^\ell_T$ we define
\[ \theta^{\bm a}_{W,T}= \sum_{v \in W} \left( \prod_{t \in T} (x_t-x_v)^{1+a_t} \right) \partial_{x_v} \in \Der(S).\]
Derivations of the form $\theta_{W,T}$ were used by Suyama and Tsujie \cite{ST} to construct an explicit basis of $D(\A_G)$ when $\A_G$ is free (see also \cite[Theorem 4.1]{NTUY}).
M\"uhlherr considered $\theta_{W,T}$ for any graph $G$ and any separator $T$ of $G$ to study generating sets of $D(\A_G)$.
The following lemma is a slight generalization of \cite[Lemma 3.4]{Mu}.
\begin{lemma}
    \label{lem:containment}
Let $G$ be a graph with vertex set $[\ell]$.
If $(W,T) \in P_G$, then
\[ \theta_{W,T}^{\bm a} \in D_R(\A_G) \ \ \text{ for all $\bm a \in \Z^\ell_T$}.\]
\end{lemma}

\begin{proof}
The coefficient of $\partial_{x_v}$ in $\theta_{W,T}^{\bm a}$ belongs to $R$.
Hence it suffices to prove $\theta_{W,T}^{\bm a} \in D(\A_G)$.
We prove $\theta_{W,T}^{\bm a} (x_u-x_v) \in (x_u-x_v)$ for every edge $\{u,v\}$ of $G$.
Since $\theta_{W,T}^{\bm a}(x_u-x_v)=0$ if $u,v \not \in W$,
we may assume $u \in W.$

Suppose $v \notin W$. Since $\{u,v\}$ is an edge of $G$ and $W$ is the vertex set of a connected component of $G-T$, we must have $v \in T$.
Then
\[ \theta_{W,T}^{\bm a} (x_u-x_v)= \prod_{t \in T} (x_t-x_u)^{1+a_t}\]
is divisible by $x_u-x_v$, as desired.

Suppose $v \in W$. Then
\[ \theta_{W,T}^{\bm a} (x_u-x_v) = \left(\prod_{t \in T} (x_t-x_u)^{1+a_t}\right) - \left( \prod_{t \in T} (x_t-x_v)^{1+a_t} \right).\]
This polynomial is zero modulo $x_u-x_v$, so 
$\theta_{W,T}^{\bm a} (x_u-x_v) \in (x_u-x_v)$ as desired.
\end{proof}

For $\theta=\sum_{k=1}^\ell f_k\partial_{x_k} \in \Der(S)$
and for $i=1,2,\dots,\ell$,
we define the element $x_i \circ \theta \in \Der(S)$ by
\[
x_i \circ \theta=\sum_{k=1}^\ell (x_i-x_k)f_k\partial_{x_k}.
\]
We note that for $1\leq i <j \leq \ell$ one has \[x_i \circ \theta -x_j \circ \theta=(x_i-x_j) \theta.\]
The next statement is an analogue of Lemma \ref{lem:obvious} for $D(\A_G)$.

\begin{lemma}
\label{technical}
Let $G$ be a graph with vertex set $[\ell]$,
$(W,T) \in P_G$, $\bm a \in \Z_T^\ell$, and let $C$ be the connected component of $G-T$ such that $V(C)=W$.
\begin{enumerate}
    \item If $i \in T$, then  $x_i \circ \theta_{W,T}^{\bm a} =  \theta_{W,T}^{\bm a+ \bm e_i}$.
    \item If $i \in W$, then one has 
    \[x_i \circ \theta_{W,T}^{\bm a} = \sum_{C'\in \mathrm{Comp}(C- \{i\})} \theta_{V(C'),T \cup \{i\}}^{\bm a}.\]
    \item If $i \in [\ell]\setminus(T \cup W)$, then $(W,T\cup\{i\}) \in P_G$ and $x_i \circ \theta_{W,T}^{\bm a}=\theta_{W,T \cup \{i\}}^{\bm a}$.
    \end{enumerate}
\end{lemma}

\begin{proof}
Parts (1) and (3) are straightforward. We prove (2).
Assume $i \in W$.
By the definition of $\theta_{W,T}^{\bm a}$, we have
\[
x_i \circ \theta_{W,T}^{\bm a}=\theta_{W \setminus \{i\},T \cup \{i\}}^{\bm a}.
\]
Since $W\setminus \{i\} = \bigcup_{C' \in \mathrm{Comp}(C-\{i\})}(V(C'))$ is a partition of $W \setminus \{i\}$,
the above element is equal to the right-hand side of (2).
\end{proof}

Recently, M\"uhlherr \cite{Mu} found an explicit generating set of $D(\A_G)$.
We recall this result.
Let $\theta_i=\sum_{k=1}^\ell x_k^i \partial_{x_k}$
for $i=0,1,\dots,\ell-1$.
The connectivity of a graph $G$ is the largest integer $k$ such that $G$ is $k$-connected.

\begin{lemma}[{M\"uhlherr \cite[Theorem 1.3]{Mu}}]
\label{mulherr}
If a graph $G$ with vertex set $[\ell]$ has connectivity $k>0$,
then
\[\{\theta_0,\theta_1,\dots,\theta_{k}\} \cup \{ \theta_{W,T} \mid (W,T) \in P_G, \ T \text{ is a separator of }G\}\]
generates $D(\A_G)$.
\end{lemma}

The following statement is a minor variation of M\"uhlherr's result,
and its proof essentially appears in \cite[Corollary 3.13 and Theorem 3.14]{Mu}.

\begin{cor}
\label{cor:generator}
Let $G$ be a graph with vertex set $[\ell]$. 
Then the set
\[ \{\theta_{W,T} \mid (W,T) \in P_G\}\]
generates $D(\A_G)$.
\end{cor}

\begin{proof}
If $\ell=1$ then the graph has no edges and $D(\A_G)=\Der(\K[x_1])$ is generated by $\partial_{x_1}=\theta_{\{1\},\emptyset}$.
Thus we may assume $\ell \geq 1$.
If $G$ is not connected, then $D(\A_G)$ is generated by the union of generators of vector fields of connected components of $G$.
Hence by the definition of $\theta_{W,T}$, to prove the statement, we may assume that $G$ is connected.
Let $D$ be the submodule of $D(\A_G)$ generated by $\{\theta_{W,T} \mid (W,T) \in P_G\}$.
By Lemma \ref{mulherr},
it suffices to prove $\theta_0,\theta_1,\dots,\theta_{\ell-1} \in D$.

We prove $\theta_k \in D$ using induction on $k$.
When $k=0$, we have $\theta_0=\theta_{[\ell],\emptyset} \in D$ since $([\ell],\emptyset) \in P_G$ by our assumption that $G$ is connected.
Assume $k>0$ and $\theta_0,\dots,\theta_{k-1} \in D$.
Observe that  for any $s \in [\ell]$ one has
\[\prod_{t=1}^k(x_t-x_s)=\sum_{i=0}^k (-1)^{i-k} e_i(x_1,\dots,x_k) x_s^{k-i}, 
\]
where $e_i(x_1,\dots,x_k)$ is the $i$th elementary symmetric function in variables $x_1,\dots,x_k$.
By the above equation, we have
\begin{align}
\label{2.10.1}    
\sum_{i=0}^k (-1)^{i-k} e_i(x_1,\dots,x_k) \theta_{k-i}=
\sum_{t=1}^\ell \left(\prod_{s=1}^k (x_s-x_t) \right)\partial_{x_t}=\theta_{[\ell]\setminus [k],[k]}.
\end{align}
Since the vertex sets of connected components of $G-[k]$ partition $[\ell]\setminus [k]$, we have
\[
\theta_{[\ell]\setminus [k],[k]}=
\sum_{C \in \mathrm{Comp(G-\setminus [k])}} \theta_{V(C),[k]}  \in D.
\]
Hence, by \eqref{2.10.1},
$\theta_k$ can be written as a $S$-linear combination of $\theta_0,\dots,\theta_{k-1}$ modulo $D$.
By the induction hypothesis, this proves $\theta_k \in D$ as desired.
\end{proof}

\begin{prop}
\label{prop:Kbasis}
Let $G$ be a graph with vertex set $[\ell]$.
Then
\begin{align}
\label{eq:Kbasis}
\{ \theta_{W,T}^{\bm a} \mid (W,T) \in P_G,\ \bm a \in \Z_T^\ell\}
\end{align}
is a $\K$-basis of $D_R(\A_G)$.
\end{prop}

\begin{proof}
We first prove that the set in \eqref{eq:Kbasis} spans $D_R(\A_G)$ as a $\K$-vector space.
By Corollary \ref{cor:generator}, to prove this, it suffices to show that for any homogeneous polynomial $f \in R$ the element $f \theta_{W,T}$ with $(W,T) \in P_G$ can be written as a $\K$-linear combination of elements in \eqref{eq:Kbasis}.
To prove this, by using induction on the degree of $f$, it is enough to show that
$(x_i-x_j) \theta_{W,T}^{\bm a}$
is a $\K$-linear combination of elements in \eqref{eq:Kbasis}
for $(W,T) \in P_G$ and $\bm a \in \Z^\ell_T$.
This last statement follows from Lemma \ref{technical} since
\begin{align*}
(x_i-x_j) \theta_{W,T}^{\bm a}
=
x_i\circ \theta_{W,T}^{\bm a}-x_j \circ\theta_{W,T}^{\bm a}
\end{align*}
and 
both $x_i\circ \theta_{W,T}^{\bm a}$
and $x_j\circ \theta_{W,T}^{\bm a}$ are $\K$-linear combinations of elements in the set \eqref{eq:Kbasis}.

It remains to prove that the set in \eqref{eq:Kbasis} is $\K$-linearly independent.
Consider the equation
\begin{align}
    \label{eq:element}
    \sum_{(W,T) \in P_G,\ \! \bm a \in \Z^\ell_T} c_{(W,T)}^{\bm a} \theta_{W,T}^{\bm a}=0,
\end{align}
where only finitely many of the coefficients $c_{(W,T)}^{\bm a} \in \K$ are nonzero.
We claim $c_{(W,T)}^{\bm a}=0$ for $(W,T) \in P_G$ and $\bm a \in \Z_T^\ell$.

It is enough to prove the claim for $(W,T) \in P_G$ and $\bm a \in \Z_T^\ell$ with $1 \in W$.
By the definition of $P_G$,
for every subset $T \subsetneq [\ell]$ with $1 \notin T$ there is a unique subset $W_T \subset [\ell]\setminus T$ such that $1 \in W_T$ and $(W_T,T) \in P_G$
(indeed, $W_T$ is the vertex set of the connected component of $G-T$ containing the vertex $1$).
Set $y_k=x_k-x_1$ for $k=2,3,\dots,\ell$.
By considering the coefficient of $\partial_{x_1}$ in \eqref{eq:element}, we obtain
\[
\sum_{1\not \in T \subsetneq [\ell]}\left(\sum_{\bm a \in \Z^\ell_T}c_{(W_T,T)}^{\bm a}y^{\bm a+\bm 1_T}\right)=0,\]
where $y^{\bm a}=y_2^{a_2} \cdots y_\ell^{a_\ell}$ for $\bm a=(a_1,\dots,a_\ell)$
and $\bm 1_T=\sum_{i \in T} \bm e_i$ is the incidence vector of $T$.
Since the monomials $y^{\bm a+\bm 1_T}$ appearing in the above equation are all distinct,
all the coefficients $c_{(W_T,T)}^{\bm a}$ must be zero.
This proves the desired linear independence.
\end{proof}

\begin{rem}
Proposition \ref{prop:Kbasis} also gives a $\K$-basis for $D(\A_G)$.
Indeed, since
\[
D(\A_G)\cong D_R(\A_G)\otimes_R S\cong D_R(\A_G)\otimes_\K \K[x_1],
\]
if $B$ is a $\K$-basis for $D_R(\A_G)$, then $\{u x_1^k \mid u \in B,\ k \geq 0\}$ is a $\K$-basis for $D(\A_G).$
\end{rem}

\subsection*{Proof of the main theorem}
We now prove Theorem \ref{thm:1.1}.
Let $G$ be a connected graph with vertex set $[\ell]$.
We regard $\K[P_G]$ as a module over $S=\K[x_1,\dots,x_\ell]$ by identifying $X_i$ with $x_i$.
By Lemma \ref{lem:Pbasis} and Proposition \ref{prop:Kbasis}, we know that
\begin{itemize}
    \item the set $\{ x_{(W,T)} X^{\bm a} \mid (W,T) \in P_G,\ \bm a \in \Z^\ell_T\}$ is a 
    $\K$-basis of $\K[P_G]$, and
    \item the set $\{ \theta_{W,T}^{\bm a} \mid (W,T) \in P_G,\ \bm a \in \Z^\ell_T\}$ is a $\K$-basis of $D_R(\A_G)$.
\end{itemize}
Let
\[
\phi_G : \K[P_G] \to D_R(\A_G)
\]
be the $\K$-linear isomorphism sending each $x_{(W,T)} X^{\bm a}$ to $\theta_{W,T}^{\bm a}$.
The following theorem completes the proof of Theorem \ref{thm:1.1}.

\begin{theorem}
The map $\phi_G$ is an isomorphism of $R$-modules.
\end{theorem}

\begin{proof}
Since $\phi_G$ is an isomorphism of $\K$-vector spaces,
what we must prove is that $\phi_G$ is $R$-linear.
For this,
it suffices to show
$$\phi_G\big( (X_i-X_j)x_{(W,T)}X^{\bm a}\big)=(x_i-x_j) \theta_{W,T}^{\bm a}$$
for all $1 \leq i <j \leq \ell$, $(W,T) \in P_G$ and $\bm a \in \Z^\ell_T$.
This follows from Lemmas \ref{lem:obvious} and \ref{technical}, which imply
\[
\phi_G (X_k x_{(W,T)}X^{\bm a})= x_k \circ \theta_{W,T}^{\bm a} \ \ \text{ for all }k \in [\ell],
\]
together with the fact that $x_i \circ \theta_{W,T}^{\bm a}-x_j \circ \theta_{W,T}^{\bm a}=
(x_i-x_j) \theta_{W,T}^{\bm a}$.
\end{proof}

\begin{rem}
\label{Smodule}
The action ``$x_i \circ$" actually define the $S$-module structure to $D_R(\A_G)$ and the above proof says that $\K[P_G]$ and $D_R(\A_G)$ are isomorphic as $S$-modules by this module structure.
We also note that since $x_1 \circ x_2 \circ \cdots x_\ell \circ \theta=0$ for any $\theta \in \Der(S)$,
the module $D_R(\A_G) \cong \K[P_G]$ is also a module over $S/(x_1\cdots x_\ell)$.
\end{rem}

\begin{rem}
\label{Rdual}
By the isomorphism \eqref{dual} in the introduction,
$\K[P_G]$ and the module $M(G)=J(G^c)/(x_1\cdots x_\ell)$ are $R$-duals each other. Indeed,
\[
\K[P_G] \cong H_{\mideal_S}^{\ell-1} (M(G))^\vee
\cong H_{\mideal_R}^{\ell-1} (M(G))^\vee \cong
\Hom_R(M(G),R).
\]
where the second isomorphism appears in \cite[\S 7.4]{AM} and the third isomorphism is the local duality.
\end{rem}

\section{Applications and related results}

We discuss some consequences of our main result, as well as some related results.

\subsection*{Hilbert series of $D(\A_G)$}
A combinatorial formula for the Hilbert series of $D(\A_G)$
was recently given in \cite[Corollary 7.13]{AM} using the isomorphism \eqref{dual}.
The $\K$-basis in Proposition \ref{prop:Kbasis} gives a different proof of this result.
Recall that
for a graded $\K$-vector space $M=\bigoplus_{i \in \Z} M_i$,
where $M_i$ is the homogeneous component of degree $i$,
its {\bf Hilbert series} is the series
\[\mathrm{Hilb}(M,t)=\sum_{i \in \Z} (\dim_\K M_i)t^i.\]

Let $G$ be a (not necessarily connected) graph with vertex set $[\ell]$.
For $(W,T) \in P_G$, we define $\rank(W,T)=|T|$.
Note that when $G$ is connected, this rank coincides with the rank of elements in the simplicial poset $P_G$ defined at the beginning of section 2.
For $k=0,1,\dots,\ell-1$, define the {\bf $f$-vector}
$f(P_G)=(f_0,f_1,\dots,f_{\ell-1})$ of $P_G$ by
\[f_k= |\{ \sigma \in P_G \mid \rank(\sigma)=k\}| \ \ \ \text{ for }k=0,1,\dots,\ell-1.\]
We also define the {\bf $h$-vector} $h(P_G)=(h_0,h_1,\dots,h_{\ell-1})$ of $P_G$ by the equation
\[\sum_{k=0}^{\ell-1} f_k t^k (1-t)^{\ell-1-k} = \sum_{k=0}^{\ell-1} h_k t^k.\]
If $G$ is connected,
then each $f_i$ counts the number of $(i-1)$-dimensional faces of the CW-complex $\Gamma(P_G)$ corresponding to $P_G$,
and 
the vectors $f(P_G)$ and $h(P_G)$ coincide with the $f$-vector and the $h$-vector of the simplicial poset $P_G$ as defined in \cite{St}\footnote{We shift the indices of $f(P)$ compared with Stanley's notation. Our $f_i$ is $f_{i-1}$ in \cite{St}.}.

\begin{cor}
\label{cor:Hilb}
Let $G$ be a connected graph with vertex set $[\ell]$
and let $h(P_G)=(h_0,h_1,\dots,h_{\ell-1})$ be the $h$-vector of $P_G$. Then
\[
\mathrm{Hilb}\big(D(\A_G),t\big)=\frac 1 {(1-t)^\ell} (h_0+h_1t+ \cdots + h_{\ell-1}t^{\ell-1}).\]
\end{cor}

\begin{proof}
Let $f(P_G)=(f_0,f_1,\dots,f_{\ell-1})$. By Proposition \ref{prop:Kbasis}, we have
\begin{align*}
\mathrm{Hilb}\big(D_R(\A_G),t\big)&=
\sum_{\sigma \in P_G} \frac {t^{\rank \sigma}}{(1-t)^{\rank \sigma}}\\
&= \frac 1 {(1-t)^{\ell-1}}\left(\sum_{k=0}^{\ell-1} f_k t^k (1-t)^{\ell-1-k}\right)\\
&=\frac 1 {(1-t)^{\ell-1}}(h_0+h_1t+\cdots+h_{\ell-1}t^{\ell-1}).
\end{align*}
Then, since $D(\A_G)\cong D_R(\A_G) \otimes_R S$, we have
\[
\mathrm{Hilb}\big(D(\A_G),t\big)
=
\frac 1{1-t} \mathrm{Hilb}\big(D_R(\A_G),t\big)
=\frac 1 {(1-t)^{\ell}}(h_0+h_1t+\cdots+h_{\ell-1}t^{\ell-1}),
\]
as desired.
\end{proof}

\begin{rem}
\label{rem3.4}
For a connected graph $G$, Corollary \ref{cor:Hilb} can be considered as a special case of \cite[Proposition 3.8]{St}, which proves that if $P$ is a simplicial poset of rank $r$ with the $h$-vector $(h_0,h_1,\dots,h_r)$ then $\mathrm{Hilb}(\K[P],t)=\frac 1{(1-t)^r} (h_0+h_1t+\cdots +h_rt^r)$.
\end{rem}

\begin{example}
Consider the case when $G=$\raisebox{-2.5mm}{\includegraphics[width=0.06\linewidth]{4cycle.pdf}}.
One can see from Figure \ref{fig1} that $f(P_G)=(1,4,8,4)$.
Then $h(P_G)=(1,1,3,-1)$ and we have 
\[\mathrm{Hilb}\big(D(\A_G),t\big)=\frac {1+t+3t^2-t^3} {(1-t)^4}.\]
\end{example}

\begin{rem}
Let $G$ be a graph.
The $f$-vector of $P_G$ is related to the {\bf subgraph component polynomial} 
\[
\textstyle
Q(G;x,y)=\sum_{W \subset [\ell]} x^{|W|}y^{|\mathrm{Comp}(G[W])|}\]
defined in \cite{TAM}.
Indeed, 
$\big(\frac {\partial} {\partial y} Q(G;x,y)\big)|_{y=1}=\sum_{k=1}^{\ell} f_{\ell-k}x^k$.
\end{rem}

\subsection*{Combinatorial topology of $P_G$}
Here, we discuss some combinatorial and topological properties of $P_G$.
For a finite set $F$, we write $\overline F$ for the set of all subsets of $F$.
We regard $\overline F$ as a simplicial poset whose partial order is the inclusion.
Clearly, the corresponding CW-complex $\Gamma(\overline F)$ is a simplex of dimension $|F|-1$.
For two simplicial posets $P_1$ and $P_2$, their {\bf join} is the simplicial poset
\[
P_1*P_2 =\{ (\sigma,\tau) \mid \sigma \in P_1,\tau \in P_2\},
\]
where $(\sigma,\tau) \geq (\sigma',\tau')$ if and only if $\sigma \geq \sigma'$ and $\tau \geq \tau'$.
For a finite poset $P$ and $\sigma \in P$,
the link of $\sigma$ in $P$ is the poset $\mathrm{lk}_P(\sigma)=\{\tau \in P \mid \tau \geq \sigma\}$.
If $P$ is a simplicial poset, then $\mathrm{lk}_P(\sigma)$ is a simplicial poset with the minimal element $\sigma$.
The following statement follows directly from the definition of $P_G$.

\begin{lemma}
    \label{lem:straightforward}
    Let $G$ be a graph with vertex set $[\ell]$ and $\sigma=(W,T) \in P_G$.
    \begin{itemize}
        \item [(1)] The maximal elements of $P_G$ are $(\{1\},[\ell]\setminus \{1\}),\dots,(\{\ell\},[\ell]\setminus \{\ell\})$.
        \item [(2)] $\mathrm{lk}_{P_G}(\sigma)$ is isomorphic to $P_{G[W]}*\overline{\big([\ell]\setminus (W\cup T)\big)}$ as a poset.        
    \end{itemize}
\end{lemma}

We first show that $P_G$ is a special type of a simplicial poset, known as a normal pseudomanifold.
Let $P$ be a connected simplicial poset,
where we say that $P$ is connected if $\Gamma(P)$ is connected.
Maximal elements of $P$ are called {\bf facets} of $P$
and $P$ is said to be {\bf pure} if all its facets have the same rank.
A pure simplicial poset of rank $r$ is said to be a {\bf pseudomanifold} if it satisfies the following two conditions:
\begin{itemize}
    \item every rank $r-1$ element is covered by at most two elements, and
    \item the dual graph of $P$ is connected.
\end{itemize}
Here, the dual graph of $P$ is the graph whose vertices are the facets of $P$ and in which two facets $\sigma$ and $\tau$ are adjacent if there is a rank $r-1$ element covered by both $\sigma$ and $\tau$.
If moreover every rank $r-1$ element is covered by exactly two elements then we call $P$ a pseudomanifold without boundary, otherwise we call $P$ a pseudomanifold with boundary.
We note that for any pseudomanifold $P$ and $F \ne \emptyset$,
the join $P*\overline F$ is a pseudomanifold with boundary.
A pseudomanifold is said to be {\bf normal} if $\mathrm{lk}_P(\sigma)$ is connected  for all $\sigma \in P$ with $\rank(\sigma)\leq r-2$.
For a normal pseudomanifold $P$ and $\sigma \in P$,
its link $\mathrm{lk}_P(\sigma)$ is again a normal pseudomanifold
and we say that $\sigma$ is an {\bf interior face} of $P$ if $\mathrm{lk}_P(\sigma)$ is a normal pseudomanifold without boundary.
We regard facets as interior faces and a rank $r-1$ element of $P$ is an interior face if and only if it is covered by two facets.
For a graph $G$,
we say that  $W \subset V(G)$ is a {\bf clique} of $G$ if $G[W]$ is a complete graph.

\begin{prop}
    \label{pmanifold}
Let $G$ be a connected graph with vertex set $[\ell]$.
\begin{itemize}
    \item[(1)] $P_G$ is a normal pseudomanifold of rank $\ell-1$ whose dual graph is $G$. 
    \item[(2)] $\sigma=(W,T) \in P_G$ is an interior face of $P_G$ if and only if $W$ is a clique of $G$ and $T=[\ell]\setminus W$.
\end{itemize}
\end{prop}

\begin{proof}
It is obvious that $P_G$ is pure and has rank $\ell-1$ by the definition of $P_G$.
Also,
every rank $\ell-2$ element $\sigma$ must be of the form either $\sigma=(\{i\},[\ell]\setminus\{i,j\})$ or $\sigma=(\{i,j\},[\ell]\setminus\{i,j\})$.
In the former case, $P_G$ has only one facet $(\{i\},[\ell]\setminus\{i\})$ that is larger than $\sigma$
and in the latter case there are exactly two facets $(\{i\},[\ell]\setminus\{i\})$ and
$(\{j\},[\ell]\setminus\{j\})$ that are larger than $\sigma$.
This proves that $P_G$ is a pseudomanifold and the dual graph of $P_G$ must be the original graph $G$.
We also note that $P_G$ has no boundary only when $G$ is the complete graph $K_\ell$.

Let $\sigma=(W,T)\in P_G$.
Lemma \ref{lem:straightforward}(2) guarantees that
$\mathrm{lk}(\sigma)\cong P_{G[W]} *\overline {\big([\ell]\setminus(W\cup T)\big)}$ is connected.
Hence $P_G$ is a normal pseudomanifold, completing the proof of (1).
Also, $\mathrm{lk}(\sigma)$ has no boundary exactly when $G[W]$ is a complete graph and $[\ell]\setminus(T\cup W)=\emptyset$.
This proves (2).
\end{proof}

\begin{rem}
As we explained in the proof,
$P_G$ is a pseudomanifold without boundary only when $G$ is a complete graph.
In that case $P_G$ is isomorphic to the face poset of the boundary of a simplex.
\end{rem}

\begin{rem}
Proposition \ref{pmanifold} gives one nontrivial consequence.
Suppose that $G$ is not a complete graph.
Then we can define the boundary $\partial P_G$ of $P_G$, which is a subposet of $P_G$ consisting of elements of $P_G$ which are not interior faces.
Let $J$ be the ideal of $\K[P_G]$ generated by $x_\sigma$ such that $\sigma$ is an interior face of $P_G$.
Then we obtain the short exact sequence 
\[0 \longrightarrow J \longrightarrow \K[P_G] \longrightarrow \K[\partial P_G] \longrightarrow 0.\]
As an $R$-module,
$\K[P_G]$ is isomorphic to $D_R(\A_G)$.
Also, in this case, 
the canonical module $\big(H^{\ell-1}_{\mideal_R}(D_R(\A_G))\big)^\vee$ of $D_R(\A_G)$ is isomorphic to 
$\Omega_R^1(\A_G)$
(see \cite[\S 7.4]{AM} for the definition of $\Omega_R^1(\A_G)$)
and it follows from \cite[Theorem 4.4]{Sa} that the canonical module $\big(H_{\mideal_R}^{\ell-1}(\K[P_G])\big)^\vee$ of $\K[P_G]$ is isomorphic to the ideal $J$,
where $\mideal_R$ is the graded maximal ideal of $R$.
Therefore, we have the following nontrivial exact sequence\footnote{Using \cite[Theorems 7.2 and 7.6]{AM} one can check that the first map in the exact sequence is given by changing $dx_i$ to $\big(\prod_{k \ne i} (x_k-x_i)\big)\partial_{x_i}$.}
\[0 \longrightarrow \Omega_R^1(\A_G) \longrightarrow D_R(\A_G)\longrightarrow \K[\partial P_G] \longrightarrow 0.\]
The exact sequence comes from Stanley--Reisner theory and it is not clear if this exact sequence has an arrangement-theoretic meaning.
\end{rem}

We finally show that the freeness of $D(\A_G)$ corresponds to the property that $\Gamma(P_G)$ is a ball (or a sphere).
It was proved by Stanley that $D(\A_G)$ is free if and only if $G$ is chordal \cite[Theorem 3.3]{ER94},
where a graph $G$ is said to be {\bf chordal} if $G$ has no induced cycle of length $\geq 4$.
We prove the following statement.

\begin{prop}
    \label{freeness}
Let $G \ne K_\ell$ be a connected graph with vertex set $[\ell]$.
The following conditions are equivalent.
\begin{itemize}
    \item[(1)] $D(\A_G)$ is a free $S$-module.
    \item[(2)] $G$ is a chordal graph.
    \item[(3)] $\Gamma(P_G)$ is homeomorphic to an $(\ell-2)$-dimensional ball.
\end{itemize}
\end{prop}

Before the proof,
we introduce some notation on graphs and simplicial posets.
Let $G=(V,E)$ be a graph.
For $v \in V$, the {\bf neighbor} of $v$ in $G$ is the set $N_G(v)=\{u \in V \mid \{u,v\} \in E\}$.
We say that a vertex $v \in V$ is {\bf simplicial} in $G$ if $G[N_G(v)]$ is a complete graph.
A sequence $v_1,\dots,v_\ell$ of the vertices of $G$ is said to be a {\bf perfect elimination ordering} of $G$ if $v_k$ is simplicial in $G[\{v_1,\dots,v_k\}]$ for $k=2,3,\dots,\ell$.
By Dirac's theorem \cite{Di}, a graph is chordal if and only if it admits a perfect elimination ordering.
Let $P$ be a simplicial poset of rank $r$.
For elements $\sigma_1,\dots,\sigma_k \in P$,
let $\langle \sigma_1,\dots,\sigma_k\rangle = \bigcup_{i=1}^k [\hat 0, \sigma_i]$
be the simplicial poset generated by $\sigma_1,\dots,\sigma_k$.
We say that $P$ is {\bf shellable} if
there is an ordering $\sigma_1,\dots,\sigma_m$ of the facets of $P$,
called a {\bf shelling} of $P$, such that
for $i=2,3,\dots, m$,
$\langle \sigma_1,\dots,\sigma_{i-1} \rangle \cap \langle \sigma _i\rangle$ is pure of rank $r-1$.
It is known that,
if a simplicial poset $P$ is a shellable pseudomanifold,
then the corresponding regular CW-complex $\Gamma(P)$ is homeomorphic to either a ball or a sphere.
See \cite[Theorem 11.4]{Bj95}.

\begin{proof}[Proof of Proposition \ref{freeness}]
The equivalence of (1) and (2) is Stanley's result.
We prove $(3) \Rightarrow (1)$.
Suppose that $\Gamma(P_G)$ is homeomorphic to an $(\ell-2)$-dimensional ball.
Then $\K[P_G]$ is a finitely generated Cohen--Macaulay $R$-module of dimension $\ell-1$ (see  \cite[Corollary 3.5]{St} and \cite[II, 4.3 Proposition]{St96}).
Since $R$ has dimension $\ell-1$,
it follows from the Auslander--Buchsbaum formula that $D_R(\A_G)\cong \K[P_G]$ has projective dimension $0$ over $R$ and hence is a free $R$-module.
Thus $D(\A_G)=D_R(\A_G)\otimes_R S$ is a free $S$-module.

It remains to prove $(2) \Rightarrow (3)$.
We may assume that $1,2,\dots,\ell$ is a perfect elimination ordering of $G$.
Let $\sigma_i=(\{i\},[\ell]\setminus \{i\})$ for all $i$.
Recall that $\sigma_1,\dots,\sigma_\ell$ are the facets of $P_G$.
We claim that the ordering $\sigma_1,\sigma_2,\dots,\sigma_\ell$ is a shelling of $P_G$.

Fix $i \in \{2,3,\dots,\ell\}$.
For each $F \subset [\ell] \setminus \{i\}$,
let $C_F$ be the unique connected component of $G-F$ that contains the vertex $i$.
Then \[
\langle \sigma_i \rangle = \big\{ \big(V(C_F),F\big) \mid F \subset [\ell] \setminus \{i\} \big\}.
\]
Let $N=N_{G[\{1,2,\dots,i-1\}}(i)$.
Then for each $u \in N$, 
since $\{i,u\}$ is an edge of $G$ we have 
\[
\big(V(C_{[\ell]\setminus \{i,u\}}),[\ell] \setminus \{i,u\} \big)= \big( \{i,u\},[\ell]\setminus \{i,u\}\big)
\]
 and we have
\begin{align}
\label{hoshihoshi}
    \langle \sigma_1,\dots,\sigma_{i-1} \rangle \cap \langle \sigma_i \rangle
    \supset  \big \langle \big ( \{i,u\}, [\ell] \setminus \{i,u\} \big) \mid u \in N \big\rangle.
\end{align}
We claim that equality holds in \eqref{hoshihoshi}.
Note that this proves that $\langle \sigma_1,\dots,\sigma_{i-1} \rangle \cap \langle \sigma_i \rangle$ is pure and completes the proof.

Observe that elements of $\langle \sigma_i \rangle$ that are not contained in the right-hand side of \eqref{hoshihoshi} are exactly elements $(V(C_F),F)$ with $F \supset N$.
Thus to prove the equality in \eqref{hoshihoshi} it suffices to prove $(V(C_N),N) \not \in \langle \sigma_1,\dots,\sigma_{i-1} \rangle.$
Also, since $(V(C_N),N) \in \langle \sigma_k \rangle$ if and only if $k$ is a vertex of $C_N$,
what we must prove is that the vertices $1,2,\dots,i-1$ are not contained in $C_N$.

Let $j \ne i$ be a vertex of $C_N$.
We prove $j \geq i$.
Take a shortest path $i=v_0,v_1,\dots,v_m=j$ from $i$ to $j$ in $C_N$ (equivalently in $G-N$).
Then we have $v_0 < v_1 < \cdots <v_m$ since if $v_{q+1} <v_q$ and $v_{q-1}<v_q$ for some $q\geq 1$ then $\{v_{q-1},v_{q+1}\}$ is an edge of $G$ because $1,2,\dots,\ell$ is a perfect elimination ordering,
contradicting the shortestness of the path.
This in particular proves $i<j$ as desired.

Thus $P_G$ is shellable. Since $G\ne K_\ell$, Proposition \ref{pmanifold} shows that $P_G$ has boundary. Hence $\Gamma(P_G)$ is not a sphere, and therefore it is homeomorphic to a ball.
\end{proof}

\subsection*{Local cohomology, projective dimension, and regularity}

Another invariant of $D(\A_G)$ that can be computed from the poset $P_G$ is its local cohomology.

For a simplicial poset $P$,
we write $\widetilde H_i(P)$ for the $i$th reduced homology group of the CW complex $\Gamma(P)$ with coefficients in $\K$.
For a graded $\K$-algebra $A$
and a finitely generated graded $A$-module $M$,
we write $H^i_{I}(M)$ for the $i$th local cohomology module with respect to the ideal $I \subset A$.
We also write $\mideal_A$ for the graded maximal ideal of $A$ and simply write $H^i(A)=H^i_{\mideal_A}(A)$.
We refer the reader to \cite{Hu} for the basics of local cohomology modules.
The following result, proved by Duval \cite[Theorem 5.5]{Duval} is a generalization of Hochster's formula for Stanley--Reisner rings of simplicial complexes \cite[II, 4.1 Theorem]{St96}.

\begin{lemma}[Duval]
\label{duvalcohomology}
Let $P$ be a simplicial poset. Then 
\[
\mathrm{Hilb}\big(H^i(\K[P]),t\big)= 
\sum_{\sigma \in P} \left(\dim_\K \widetilde H_{i-\rank(\sigma)-1}\big(\mathrm{lk}_P(\sigma)\big)\right) \frac {t^{-\rank(\sigma)}} {(1-t^{-1})^{\rank (\sigma)}}
\]
for all $i$.
\end{lemma}

We emphasize that the local cohomology modules of face rings have no nonzero graded components in positive degrees, so their Hilbert series are expanded in negative powers of $t$.

Let $G$ be a connected graph with vertex set $[\ell]$
and let $P_G$ be the corresponding simplicial poset.
We already know that $\K[P_G]$ is a finitely generated module over $R=\K[x_2-x_1,\dots,x_\ell-x_1]$,
where $x_i$ is identified with $\sum_{\sigma \in P_G, \mathrm{supp}(\sigma)=\{i\}} x_\sigma$,
so $\K[P_G]$ is integral over $R$ and the radical of $(x_2-x_1,\dots,x_\ell-x_1)\K[P_G]$ is equal to the maximal ideal of $\K[P_G]$ (see e.g.\ \cite[Proposition 5.6 and Corollary 5.8]{AM}). This guarantees
\[
H_{\mideal_R}^i\big(\K[P_G]\big) \cong H^i \big(\K[P_G]\big)
\]
as $R$-modules (see \cite[Propositions 2.13 and 2.14]{Hu}).
Since 
\[D(\A_G) \cong D_R(\A_G)\otimes_R S \cong D_R(\A_G) \otimes_{\K} \K[x_1],\]
by the K\"unneth formula, we have
\[H_{\mideal_S}^i \big(D(\A_G)\big) \cong 
H_{\mideal_R}^{i-1} \big(D_R(\A_G)\big) \!\otimes_\K\! H^1\big(\K[x_1]\big)
\cong
H_{\mideal_R}^{i-1} (D_R(\A_G))\! \otimes_\K\! \big(x_1^{-1} \K[x_1^{-1}]\big).\]
Since $D_R(\A_G) \cong \K[P_G]$,
by Lemma \ref{duvalcohomology} we get the following formula.

\begin{cor}
\label{cor:localcohomology}
Let $G$ be a connected graph with vertex set $[\ell]$. Then
\[
\mathrm{Hilb}(H^i_{\mideal_S}\big(D\big(\A_G)\big),t\big)=
\sum_{\sigma \in P_G} \left(\dim_\K \widetilde H_{i-\rank(\sigma)-2}\big(\mathrm{lk}_{P_G}(\sigma)\big) \right)\frac {t^{-\rank(\sigma)-1}} {(1-t^{-1})^{\rank (\sigma)+1}}\]
for all $i$.
\end{cor}

\begin{rem}
We use the convention that, 
for a simplicial poset $P \ne \{\hat 0\}$,
we have $\widetilde H_{i}(P)=0$ for $i<0$,
while $\widetilde H_{-1}(\{\hat 0\}) \cong \K$.
In Corollary \ref{cor:localcohomology},
since $\mathrm{lk}_{P_G}(\sigma)$ is connected, we also have
$\widetilde H_{0}(\mathrm{lk}_{P_G}(\sigma))=0$.
Consequently,
when $\ell \geq 3$ 
we have $H^i_{\mideal_S}(D(\A_G))=0$ for $i=0,1,2$. This reflects the fact that $\pd_S(D(\A_G)) \leq \ell-3.$
\end{rem}

Recall that for a finitely generated graded $S$-module $M$,
the {\bf depth} of $M$ is the number
\[
\depth(M)=\min\{ i \mid H_{\mideal_S}^i(M) \ne 0\}
\]
and the {\bf (Castelnuovo--Mumford) regularity} of $M$ is the number
\[
\mathrm{reg}(M)=\max\{i+j\mid H^i_{\mideal_S}(M)_j \ne 0\}.
\]
Recall that by the Auslander--Buchsbaum formula,
the projective dimension $\mathrm{pd}_S(M)$ of $M$ equals $\ell-\mathrm{depth}(M)$.
By Corollary \ref{cor:localcohomology} we also get the following statement.

\begin{cor}
\label{cor:pdandreg}
Let $G$ be a connected graph with vertex set $[\ell]$ and let $r$ be a positive integer. Then
\begin{itemize}
    \item[(1)] $\pd_S\!\big(D(\A_G)\big) \!\leq\! r$ $\Leftrightarrow$
    $\widetilde H_{i}\big(\mathrm{lk}_{P_G}(\sigma)\big)\!=\!0 $ for all $i\!<\!\ell\!-\!r\!-\!\rank(\sigma)\!-\!2$ and $\sigma \!\in\! P_G$.
    \item[(2)] $\reg(D(\A_G)) \leq r$ $\Leftrightarrow$
    $\widetilde H_{i}(\mathrm{lk}_{P_G}\big(\sigma)\big)=0 $ for all $i\geq r$ and $\sigma \in P_G$.
\end{itemize}
\end{cor}

\begin{example}
By Corollary \ref{cor:pdandreg}(1),
we have $\pd_S(D(\A_G))=\ell-3$ if and only if 
$\widetilde H_1(P_G) \ne 0$ 
when $\ell \ge 4$.
Also, 
by Corollary \ref{cor:pdandreg}(2),
if $\widetilde H_{\ell-2}(P_G)=0$ and $\widetilde H_{\ell-3}(P_G)\ne 0$, then we have $\reg(D(\A_G))=\ell-2$.

Let $G$ be a cycle of length 4.
Then the CW complex $\Gamma(P_G)$ is homeomorphic to the band $S^1 \times I$ (see Figure \ref{fig1}), so $\widetilde H_1(P_G) \ne 0$ and $\widetilde H_2(P_G)=0$.
Hence $\pd_S(D(\A_G))=1$ and $\reg(D(\A_G))=2$.
\end{example}

We say that $G$ is \textbf{weakly chordal} if neither $G$ nor its complement $G^c$ contains an induced cycle of length at least $5$.
Note that chordal graphs are weakly chordal.
Abe--K\"uhne--M\"ucksch--M\"uhlherr \cite{AKMM} recently proved that
$\mathrm{pd}_S(D(\A_G)) \leq 1$ if and only if $G$ is weakly chordal.
Considering this result,
it would be interesting to seek a combinatorial characterization
for the condition $\mathrm{pd}_S(D(\A_G)) \leq r$ for any $r$ using Corollary \ref{cor:pdandreg}.
At present, even the following consequence of \cite{AKMM} seems nontrivial.

\begin{cor}
\label{weaklychordal}
Let $G$ be a connected graph.
Then
$G$ is weakly chordal if and only if
$
\widetilde{H}_{i}(\mathrm{lk}_{P_G}(\sigma))= 0$
for all $i <\ell-3-\rank(\sigma)$ and $\sigma \in P_G$.
\end{cor}

Corollary \ref{cor:pdandreg} indeed has some consequences for the projective dimension of $D(\A_G)$.
For example,
the following statement immediately follows from the corollary.

\begin{prop}
Let $G$ be a $k$-connected graph with vertex set $[\ell]$. Then 
\[\pd_S(D(\A_G)) \leq \ell-k-1.\]
\end{prop}

\begin{proof}
Fix $\sigma \in P_G$ with $\rank(\sigma) \leq k-1$.
By Corollary \ref{cor:pdandreg},
to prove the proposition what we must prove is
\[
\widetilde H_i\big(\mathrm{lk}_{P_G}(\sigma)\big)=0
\ \ 
\text{ for all $i<k-1-\rank(\sigma)$.}
\]
Recall that the $j$-skeleton $\mathrm{skel}_j(\Gamma)$ of a CW-complex $\Gamma$ is the subcomplex of $\Gamma$ consisting of faces of $\Gamma$ of dimension $\leq j$.
The $k$-connectedness of $G$ implies that 
\begin{align}
\label{laststatement}    
\{(W,T) \in P_G \mid |T| <k\}=\{ ([\ell] \setminus T,T) \mid T \subset [\ell] \text{ and } |T| <k\}.
\end{align}
We denote by $\Delta^d=\Gamma(\overline{\{1,2,\dots,d\}})$ a simplex of dimension $(d-1)$.
By the above equation,
the $(k-2)$-skeleton of $\Gamma(P_G)$ coincides with the $(k-2)$-skeleton of the $(\ell-1)$-dimensional simplex $\Delta^{\ell}$.
Moreover, since there is at least one element $\sigma \in P_G$ with $\mathrm{supp}(\sigma)=T$ for each $T \subsetneq [\ell]$,
the complex $\Gamma(P_G)$ contains the $(k-1)$-skeleton of an $(\ell-1)$-dimensional simplex as a subcomplex.
Considering  $\mathrm{lk}_{P_G}(\sigma)$,
it follows that
$\mathrm{skel}_{k-2-\rank(\sigma)}(\Gamma(\mathrm{lk}_{P_G}(\sigma)))$ coincides with $\mathrm{skel}_{k-2-\rank(\sigma)}(\Delta^{\ell-1-\rank(\sigma)})$,
and we may consider that $\mathrm{lk}_{P_G}(\sigma)$ contains the $(k-1-\rank (\sigma))$-skeleton of an $\Delta^{\ell-1-\rank(\sigma)}$ as a subcomplex.
This proves
\begin{itemize}
    \item $\widetilde H_i\big(\mathrm{lk}_{P_G}(\sigma)\big)\cong \widetilde H_i\big( \Delta^{\ell-1-\rank(\sigma)}\big)=0$ for $i< k-2 -\rank(\sigma)$, and
    \item $\widetilde H_{k-2 -\rank(\sigma)}(\mathrm{lk}_{P_G}\big(\sigma)\big)$ is a quotient space of $\widetilde H_{k-2 -\rank(\sigma)}\big(\Delta^{\ell-1-\rank(\sigma)}\big)=0$.
\end{itemize}
(Recall $k<\ell$ by the definition of the connectivity).
Hence we have the desired equation
\eqref{laststatement}.
\end{proof}

\end{document}